\documentclass[brochure,english,12pt]{bourbaki}
\usepackage{amsmath}
\usepackage{amssymb}
\usepackage{pstricks}
\usepackage{verbatim}
\usepackage{caption}

\usepackage{tikz, tikz-cd}
\usetikzlibrary{matrix,calc,positioning,arrows,decorations.pathreplacing,decorations.markings,patterns}
\tikzset{->-/.style={decoration={
			markings,
			mark=at position #1 with {\arrow{>}}},postaction={decorate}}}
\tikzset{-<-/.style={decoration={
					markings,
					mark=at position #1 with {\arrow{<}}},postaction={decorate}}}

\usepackage{wrapfig}
\usepackage[english]{babel}
\usepackage{microtype}

\newtheorem{statement}[defi]{Statement}

\newcommand{\bC}{\mathbb{C}}

\newcommand{\bF}{\mathbb{F}}

\newcommand{\bL}{\mathbb{L}}

\newcommand{\bN}{\mathbb{N}}

\newcommand{\bQ}{\mathbb{Q}}
\newcommand{\bR}{\mathbb{R}}

\newcommand{\bZ}{\mathbb{Z}}

\newcommand\lra{\longrightarrow}

\DeclareMathOperator*{\colim}{colim}

\newcommand{\Frob}{\mathrm{Frob}}
\newcommand{\et}{\text{{\'e}t}}
\newcommand{\Hur}{\mathrm{Hur}}
\newcommand{\CHur}{\mathrm{CHur}}
\newcommand{\hcoker}{/\!\!/}

\mathchardef\ordinarycolon\mathcode`\:
\mathcode`\:=\string"8000
\begingroup \catcode`\:=\active
  \gdef:{\mathrel{\mathop\ordinarycolon}}
\endgroup

\renewcommand{\leq}{\leqslant}
\renewcommand{\geq}{\geqslant}

\date{Juin 2019}
\bbkannee{71\textsuperscript{e} ann\'ee, 2018--2019}
\bbknumero{1162}
\title[Homology of Hurwitz spaces]{Homology of Hurwitz spaces\\and\\the Cohen--Lenstra heuristic for function fields}
\subtitle{after Ellenberg, Venkatesh, and Westerland}

\author{Oscar Randal-Williams}
\address{University of Cambridge\\
Centre for Mathematical Sciences\\
Wilberforce Road\\
Cambridge CB3 0WB\\
United Kingdom}
\email{or257@cam.ac.uk}

\begin{document}

\maketitle


\section*{Introduction}
\label{sec:introduction}

Ellenberg and Venkatesh~\cite{EV} introduced the idea of analysing the function field analogue of the conjecture of Malle (on the distribution of number fields with given Galois group $G$) by relating Malle's conjectural upper bound with the asymptotics of $\bF_q$-point counts on Hurwitz schemes~$\mathsf{H}_{G,n}$. Under the heuristic that each $\bF_q$-rational component of $\mathsf{H}_{G,n}$ contains $q^n$ $\bF_q$-points they were able to precisely reproduce the upper bound in Malle's conjecture.

In a breakthrough paper, Ellenberg, Venkatesh, and Westerland~\cite{EVW} applied similar reasoning to relate the function field analogue of the Cohen--Lenstra heuristic (on the distribution of imaginary number fields with $\ell$-parts of their class groups isomorphic to a fixed group $A$) with the asymptotics of $\bF_q$-point counts on certain Hurwitz schemes~$\mathsf{Hn}_{G,n}^c$, with $G = A \rtimes \bZ^\times$ a generalised dihedral group and $c \subset G$ the conjugacy class of involutions. In this case they were---remarkably---able to \emph{justify} the heuristic that each $\bF_q$-rational component of $\mathsf{Hn}_{G,n}^c$ contains $q^n$ $\bF_q$-points, by using the Grothendieck--Lefschetz trace formula and a comparison between {\'e}tale and singular cohomology to reduce it to a problem in algebraic topology, and then solving this problem.

The topological problem concerns the singular homology of the corresponding Hurwitz spaces $\mathsf{Hn}_{G,n}^c(\bC)^{an}$. It is easy to show that the total dimension of the homology of these spaces is at most $(2|c|)^n$, but in order to show that the main term in the Grothendieck--Lefschetz trace formula is not overwhelmed as $n \to \infty$ one must show that there is not too much homology in homological degrees which are small compared with $n$. Ellenberg, Venkatesh, and Westerland accomplish this by proving a \emph{homological stability theorem} for these Hurwitz spaces.

The phenomenon of homological stability was discovered by Quillen, in his analysis of the homology of general linear groups in relation to algebraic $K$-theory. This is the phenomenon that for many natural sequences of spaces
$$X_1 \lra X_2 \lra X_3 \lra \cdots$$
the induced maps $H_d(X_{n-1}) \to H_d(X_n)$ are isomorphisms as long as $d \leq f(n)$, for some divergent function $f$. In this case $H_d(X_n)$ agrees with the direct limit $H_d(X_\infty) = \colim_i H_d(X_i)$ for all but finitely many $n$. There is a remarkable range of families $\{X_n\}$ for which homological stability is known to hold. When the $X_n$ are Eilenberg--MacLane spaces for groups $G_n$, one may take: symmetric groups, braid groups, general linear groups over rings of finite stable rank~\cite{vdK}, mapping class groups of surfaces~\cite{Harer}, automorphism groups of free groups~\cite{HatcherAutFn}, Higman--Thompson groups~\cite{SzymikWahl}, Coxeter groups~\cite{HepworthCox}, and many others. When the $X_n$ are moduli spaces, broadly interpreted, one may take: configuration spaces~\cite{McDuff}, classifying spaces for fibre bundles~\cite{GRWHomStabI}, classifying spaces for fibrations or block bundles~\cite{BM}, and many variants of these. 

In a related direction, the more recent development of \emph{representation
  stability}~\cite{CF} ---in which there is a sequence of groups $\Gamma_n$
acting on the $X_n$ in a compatible manner and the eventual behaviour of
$H_d(X_n)$ as a $\Gamma_n$-representation is studied--- may be applied to study asymptotics of weighted point counts (i.e.\ moments) of sequences of schemes, cf.~\cite{CEFVar}.

There is a more or less standard pattern in most proofs of homological stability,\footnote{Though Galatius, Kupers, and I have recently proposed another~\cite{GKR-W}, which in fortunate circumstances can provide information beyond classical homological stability, for example in the case of mapping class groups of surfaces~\cite{GKR-W2}.} in which one constructs an approximation to $X_n$ from the spaces $\{X_i\}_{i<n}$ in a standard way (cf.~\cite{R-WW}), and then is left with the problem of proving that it is a good approximation, which invariably leads one to analyse the connectivity of certain simplicial complexes associated to the situation in hand. Ellenberg, Venkatesh, and Westerland follow this general strategy, but because the Hurwitz spaces $\mathsf{Hn}_{G,n}^c(\bC)^{an}$ are disconnected a new kind of difficulty arises. To surmount this difficulty, they invent a clever piece of homological algebra.

In this exposition of~\cite{EVW} I will present their argument differently to the way it appears in that paper, closer to the framework of~\cite{GKR-W} than to the classical approach to homological stability described above. While many of the key steps are unchanged, I find that this streamlined argument clarifies the essential points.

\section{The Cohen--Lenstra heuristic for function fields}\label{sec:CLHeur}

Let $\ell$ be a prime number. The \emph{Cohen--Lenstra distribution} is the probability measure~$\mu$ on the set of isomorphism classes of finite abelian $\ell$-groups given by
$$\mu(A) = \frac{\prod_{i\geq 1}(1-\ell^{-i})}{|\mathrm{Aut}(A)|}.$$
The numerator is simply a normalisation to make $\mu$ into a probability measure: what is important is that an abelian $\ell$-group is counted with weight proportional to the reciprocal of the size of its automorphism group, as one does in the cardinality of groupoids. 

The original Cohen--Lenstra heuristic~\cite{CL} suggests that when $\ell$ is odd the $\ell$-part of the class groups of imaginary quadratic extensions of $\bQ$ is distributed according to $\mu$. The analogue for function fields was first considered by Friedman and Washington~\cite{FW}. In the function field case, $K = \bF_q(t)$, a quadratic extension $L \supset K$ is called \emph{imaginary} if it is ramified at infinity, or equivalently if it is of the form $K(\sqrt{f(t)})$ with $f$ a squarefree polynomial of odd degree $n$. 

For $n$ odd let $\mathfrak{S}_n$ denote the set of such imaginary quadratic extensions $L \supset K$  up to $K$-isomorphism, and for a fixed finite abelian $\ell$-group $A$ let $\iota : \mathfrak{S}_n \to \{0,1\}$ denote the indicator function for those $L$ with $\ell$-part of their class group isomorphic to $A$. Define the upper and lower densities
\begin{align*}
\delta^+(q) := \limsup_{n \to \infty}\frac{\sum_{L \in \mathfrak{S}_n} \iota(L)}{|\mathfrak{S}_n|} \quad\text{ and }\quad \delta^-(q) := \liminf_{n \to \infty}\frac{\sum_{L \in \mathfrak{S}_n} \iota(L)}{|\mathfrak{S}_n|}.
\end{align*}
The formulation of the Cohen--Lenstra heuristic proved by Ellenberg, Venkatesh, and Westerland is as follows, where a prime power $q$ is called \emph{good for $\ell$} if $q$ is odd and neither $q$ nor $q-1$ is divisible by $\ell$.

\begin{theo}[Ellenberg--Venkatesh--Westerland]\label{thm:Main}
Suppose $\ell$ is odd. As $q \to \infty$ with $q$ good for $\ell$, both $\delta^+(q)$ and $\delta^-(q)$ converge to $\mu(A)$.
\end{theo}

I will mainly discuss the solution of the topological problem that Theorem~\ref{thm:Main} reduces to, but will first briefly outline how this topological problem arises.

\section{Reduction to counting points of Hurwitz schemes}

The key property of the Cohen--Lenstra distribution $\mu$ is that for any finite abelian $\ell$-group $A$ the expected number of surjections $A' \to A$ is 1 when $A'$ is distributed according to $\mu$, and in fact this property characterises $\mu$~\cite[Lemma 8.2]{EVW}. For $L \in \mathfrak{S}_n$ one writes $m_A(L)$ for the number of surjections from the class group of $L$ to $A$. Using the above characterisation of the measure $\mu$, Ellenberg, Venkatesh, and Westerland show~\cite[p.\ 777]{EVW} that to prove Theorem~\ref{thm:Main} it suffices to prove the following.
 
\begin{theo}[Theorem 8.8 of~\cite{EVW}]\label{thm:EVW8point8}
Suppose $\ell$ is odd and $q$ is good for $\ell$. There is a constant $B=B(A)$ such that
$$\left|\frac{\sum_{L \in \mathfrak{S}_n} m_A(L)}{|\mathfrak{S}_n|}-1 \right| \leq \frac{B}{\sqrt{q}}$$
for all $n$ and $q$ with $\sqrt{q} > B$, $n > B$, and $n$ odd.
\end{theo}

A pair $(G,c)$ of a finite group $G$ and a conjugation-invariant subset $c \subset G$ is called \emph{admissible} if $c$ generates $G$ and if whenever $g \in c$ then $g^n \in c$ for all $n$ coprime to~$|G|$.  If $(G,c)$ is admissible then there are \emph{Hurwitz schemes} $\mathsf{Hn}_{G,n}^c$ over $\mathsf{Spec}(\bZ[1/|G|])$ which parametrise connected branched Galois $G$-covers of the affine line with $n$ branch points and monodromy in the class $c$. These schemes are formed out of similar schemes $\mathsf{H}_{G,n}$ parametrising branched Galois $G$-covers of the projective line, which have been constructed by Romagny and Wewers~\cite{RomagnyWewers}.

The crucial relation between the Cohen--Lenstra heuristic for function fields and these Hurwitz schemes, which was discovered by Yu~\cite{Yu}, is as follows. For an odd prime $\ell$ and an abelian $\ell$-group $A$, form the semi-direct product $G = A \rtimes \bZ^\times$, where $\bZ^\times$ acts on~$A$ by inversion, and let $c \subset G$ denote the conjugacy class of involutions. The pair $(G,c)$ is admissible in the sense defined above, and using Proposition 8.7 of~\cite{EVW}, which for this choice of $(G,c)$ relates surjections $\mathrm{Cl}(\mathcal{O}_L) \to A$ to branched Galois $G$-covers of the affine line $\mathsf{A}^1_{\bF_q}$ with monodromy in the class $c$, one shows that
$$\sum_{L \in \mathfrak{S}_n} m_A(L) = 2\cdot\#\mathsf{Hn}_{G,n}^c(\bF_q).$$
On the other hand, the number of squarefree polynomials of degree $n$ is $(q-1)(q^n-q^{n-1})$, but the sets of $\frac{q-1}{2}$ polynomials which differ only by a square in $\bF_q^\times$ define isomorphic quadratic extensions, so
$$|\mathfrak{S}_n| = 2 \cdot (q^n - q^{n-1}).$$
To prove Theorem~\ref{thm:Main} it therefore suffices to prove the following.

\begin{statement}\label{state:PtCountEstimate}
Suppose $\ell$ is odd and $q$ is good for $\ell$. There is a constant $B=B(A)$ such that
$$\left|\frac{\#\mathsf{Hn}_{G,n}^c(\bF_q)}{q^n}-1 \right| \leq \frac{B}{\sqrt{q}}$$
for all $n$ and $q$ with $\sqrt{q} > B$, $n > B$, and $n$ odd.
\end{statement}

\section{Point counting and homological stability}

\subsection{Example of the method: squarefree polynomials} 

I will first illustrate how algebraic topology may be used to prove results such as Statement~\ref{state:PtCountEstimate} with a much simpler example. The squarefree, monic, degree $n$ polynomials over $\bF_q$ are the $\bF_q$-points $\mathsf{C}_n(\bF_q)$ of a scheme $\mathsf{C}_n$ over $\mathsf{Spec}(\bZ)$. Parametrising monic degree $n$ polynomials by their $n$ coefficients, $\mathsf{C}_n$ may be described as the complement in~$\mathsf{A}^n$ of the zero-locus of the discriminant morphism $\Delta: \mathsf{A}^n \to \mathsf{A^1}$.

As a squarefree, monic, degree $n$ polynomial over~$\bC$ is determined by its unordered set of $n$ distinct roots, the set of complex points in the analytic topology $\mathsf{C}_n(\bC)^{an}$ is precisely the space of configurations of $n$ distinct unordered points in $\bC$. This space is well-studied\footnote{Its fundamental group is Artin's braid group on $n$ strands, $\beta_n$, and in fact $\mathsf{C}_n(\bC)^{an}$ is an Eilenberg--MacLane space for this group. The homology of this space therefore coincides with the group homology of $\beta_n$, and can also be studied from this perspective.} and its $\bQ$-homology can be computed by many methods (originally by Arnol'd~\cite{Arnold}): for $n \geq 2$ it is
\begin{equation}\label{eq:HomologyBraid}
H_i(\mathsf{C}_n(\bC)^{an};\bQ) \cong \begin{cases}
\bQ & \text{ if $i$ is 0 or 1}\\
0 & \text{ otherwise}.
\end{cases}
\end{equation}
Furthermore, the discriminant restricts to a morphism $\Delta : \mathsf{C}_n \to \mathsf{A^1} \setminus \{0\}$ which on complex points gives a continuous map $\mathsf{C}_n(\bC)^{an} \to \bC^\times$, and this map induces the above isomorphism on $\bQ$-homology.

To evaluate the number $\# \mathsf{C}_n(\bF_q)$ of squarefree, monic, degree $n$ polynomials over~$\bF_q$ one may try to apply the Grothendieck--Lefschetz trace formula to the smooth \mbox{$n$-dimensional} scheme $\mathsf{C}_n$, in the form
$$\# \mathsf{C}_n(\bF_q) = q^n \sum_{i=0}^n (-1)^i\mathrm{Tr}(\Frob_q : H^i_{\et}(\mathsf{C}_n /\overline{\bF}_q;\bQ_r)^\vee \circlearrowleft)$$
where $r$ is an auxiliary prime number not dividing $q$. If there were a natural comparison isomorphism
\begin{equation}\label{eq:ArtinComp}
H^i_{\et}(\mathsf{C}_n /\overline{\bF}_q;\bQ_r) \overset{\sim}\lra H^i(\mathsf{C}_n(\bC)^{an} ;\bQ_r),
\end{equation}
then by the isomorphism \eqref{eq:HomologyBraid} and the fact that it is induced by the discriminant morphism, one would be able to calculate the {\'e}tale cohomology of $\mathsf{C}_n /\overline{\bF}_q$, as well as the action of $\Frob_q$, to be
\begin{equation*}
H^i_{\et}(\mathsf{C}_n /\overline{\bF}_q;\bQ_r) \cong \begin{cases}
\bQ_r(0) & \text{ if $i=0$}\\
\bQ_r(-1) & \text{ if $i=1$}\\
0 & \text{ otherwise}.
\end{cases}
\end{equation*}
This recovers the well-known calculation $\# \mathsf{C}_n(\bF_q) = q^n - q^{n-1}$.

The only step I have missed is the comparison isomorphism \eqref{eq:ArtinComp}. There is indeed such an isomorphism, which in this situation has been established by Lehrer~\cite{Lehrer}.

\vspace{2ex} 

That one knows all the homology of braid groups as in \eqref{eq:HomologyBraid} is a luxury that will not be available in typical applications of this method, so let me revisit the above example assuming only a more representative amount of information. Using topological methods one may hope to establish an analogue of the following.

\begin{theo}[Homological stability]\label{thm:BraidHomStab}
There are maps
$$H_i(\mathsf{C}_{n-1}(\bC)^{an} ; \bQ) \lra H_i(\mathsf{C}_n(\bC)^{an} ; \bQ),$$
which are isomorphisms as long as $2i < n$.
\end{theo}

As I described in the introduction, there is a large literature on such homological stability theorems. By very different topological techniques, one may also hope to establish an analogue of the following.

\begin{theo}[Stable homology]\label{thm:BraidStabHom}
There is an isomorphism
$$\colim\limits_{n \to \infty} H_i(\mathsf{C}_n(\bC)^{an} ; \bQ) \cong \begin{cases}
\bQ & \text{ if $i$ is 0 or 1}\\
0 & \text{ otherwise},
\end{cases}$$
where the direct limit is formed using the maps in Theorem~\ref{thm:BraidHomStab}.
\end{theo}

There is also a significant literature on identifying the stable homology of sequences of spaces, the most significant in recent years being the identification of the stable homology of mapping class groups by Madsen and Weiss~\cite{MW}, the identification of the stable homology of $\mathrm{Aut}(F_n)$ by Galatius~\cite{GalatiusAutFn}, the identification of the stable homology of diffeomorphism groups of certain high-dimensional manifolds by Galatius and myself~\cite{GRWHomStabII}, and the identification of the homology of Higman--Thompson groups by Szymik and Wahl~\cite{SzymikWahl}. 

In addition, one may hope to estimate the complexity of the spaces in question, analogous to the following, which has been proved by Fox and Neuwirth~\cite{FN} and by Fuks~\cite{Fuks}.

\begin{theo}[Complexity]\label{thm:BraidComplexity}
The space $\mathsf{C}_n(\bC)^{an}$ has the homotopy type of a cell complex having $2^n$ cells.
\end{theo}

\subsubsection{The conclusion using Deligne's bounds} Using the Grothendieck--Lefschetz trace formula together with Deligne's theorem that eigenvalues of $\Frob_q$ acting on $H^i_{\et}(\mathsf{C}_n /\overline{\bF}_q;\bQ_r)^\vee \cong H^{2n-i}_{\et, c}(\mathsf{C}_n /\overline{\bF}_q;\bQ_r)(-n)$ are bounded above in absolute value by $q^{-i/2}$, one finds that for each fixed $n$ the limit
$$\lim_{q \to \infty} \frac{\# \mathsf{C}_n(\bF_q)}{q^n}$$
is equal to $\lim_{q \to \infty}\mathrm{Tr}(\Frob_q : H^0_{\et}(\mathsf{C}_n /\overline{\bF}_q;\bQ_r) \circlearrowleft)$ if either limit exists. The latter does and is equal to 1, as $\mathsf{C}_n$ has a single geometrically connected component, which is defined over any $\bF_q$.

\subsubsection{The conclusion using stability} The limit over $q$ formed above is often not so interesting, and instead one would like to know whether the limit
$$\lim_{n \to \infty} \frac{\# \mathsf{C}_n(\bF_q)}{q^n}$$
exists and what it is. Suppose that one only knows the Homological Stability Theorem and the Complexity Theorem. For $n \geq 2i$ one may then estimate
$$\dim_\bQ H^i(\mathsf{C}_n(\bC)^{an};\bQ) = \dim_\bQ H^i(\mathsf{C}_{2i}(\bC)^{an};\bQ) \leq 2^{2i}.$$
This is not a very refined estimate but it is at least independent of $n$. Returning to the Grothendieck--Lefschetz trace formula, and using that $\mathsf{C}_n$ has a single geometrically connected component which is $\bF_q$-rational for any $q$, one finds that
\begin{equation}\label{eq:HomStabOnlyEstimate}
\left| \frac{\# \mathsf{C}_n(\bF_q)}{q^n} - 1 \right| \leq \sum_{i=1}^{\lfloor n/2 \rfloor} q^{-i/2} 2^{2i} + \sum_{i=\lfloor n/2 \rfloor+1}^n q^{-i/2} 2^{2n},
\end{equation}
where the first term is a geometric sum and the second term is less than $(\frac{2^2}{q^{1/4}})^n(n/2)$. Thus, writing
\begin{equation*}
\delta^+(q) := \limsup_{n \to \infty} \frac{\# \mathsf{C}_n(\bF_q)}{q^n} \quad\text{ and }\quad
\delta^-(q) := \liminf_{n \to \infty} \frac{\# \mathsf{C}_n(\bF_q)}{q^n}
\end{equation*}
for the upper and lower densities, as long as $q > 2^8$ one obtains the estimate
$$|\delta^\pm(q) - 1| \leq \frac{2^2}{q^{1/2}-2^2}.$$
This in particular tells us that $\delta^\pm(q) \to 1$ as $q \to \infty$, which does not follow from the estimate using Deligne's bound alone (which would give $\lim_{n\to \infty} \lim_{q \to \infty}$ rather than $\lim_{q\to \infty} \lim_{n \to \infty}$). However, it does not tell us whether $\frac{\# \mathsf{C}_n(\bF_q)}{q^n}$ tends to a limit with $n$ for any particular $q$.

It is this style of argument that will be used to prove Theorem~\ref{thm:Main}.

\subsubsection{The conclusion using stability and stable homology}\label{sec:StabHomArg} If in addition one knows the Stable Homology Theorem then one can replace the first term in the estimate \eqref{eq:HomStabOnlyEstimate} with a correction to the left-hand side, giving instead
\begin{equation}
\left| \frac{\# \mathsf{C}_n(\bF_q)}{q^n} - (1-q^{-1}) \right| \leq  \sum_{i=\lfloor n/2 \rfloor+1}^n q^{-i/2} 2^{2n}
\end{equation}
and so $\delta^\pm(q)=1-q^{-1}$ as long as $q > 2^8$. In particular the upper and lower densities agree, so $\frac{\# \mathsf{C}_n(\bF_q)}{q^n}$ converges as $n \to \infty$ to the limit $1-q^{-1}$ for all large enough $q$.

\subsection{The case of Hurwitz schemes}

Ellenberg, Venkatesh, and Westerland propose to use the above method to prove Statement~\ref{state:PtCountEstimate}, and hence Theorem~\ref{thm:Main}. In this case one should choose a sufficiently large auxiliary prime number $r$ coprime to $q$. Supposing that an appropriate comparison theorem is available, which is established in Proposition 7.7 of~\cite{EVW}, Statement~\ref{state:PtCountEstimate} can be proved by the above method by showing:
\begin{enumerate}
\item That $\mathrm{Tr}(\Frob_q : H^0_{\et}(\mathsf{Hn}_{G,n}^c /\overline{\bF}_q;\bQ_r) \circlearrowleft) = 1$, or in other words that $\mathsf{Hn}_{G,n}^c$ has just one geometrically connected component which is $\bF_q$-rational. (As $\mathsf{Hn}_{G,n}^c$ has many connected components this is of course not possible for all $q$. It is here that the assumption that $\ell$ does not divide $q-1$ enters: all connected components but one contain $\ell$-th roots of unity in their field of definition.\footnote{This means that the argument explained here implies that the Cohen--Lenstra heuristic is not valid for function fields $\bF_q(t)$ when $\bF_q$ contains $\ell$-th roots of unity, but on the other hand it also explains how it should be corrected. This correction has been implemented in some cases by Garton~\cite{Garton}. That a correction to the Cohen--Lenstra heuristic in the number field case may be necessary in the presence of $\ell$-th roots of unity had been observed earlier by Malle~\cite{MalleCorrect}.}) I will not discuss this point any further.

\item That there are constants $C_0$ and $C_1$, depending only on the abelian group~$A$, such that
\begin{equation*}
\dim_{\bQ} H^i(\mathsf{Hn}_{G,n}^c(\bC)^{an};\bQ) \leq C_0 \cdot C_1^i
\end{equation*}
for all large enough $n$.
\end{enumerate}

In turn, the second item would follow if there were a Complexity Theorem of the form $\dim_{\bQ} H^i(\mathsf{Hn}_{G,n}^c(\bC)^{an};\bQ) \leq D_0^n$, for a constant $D_0$ depending only on the abelian group $A$, and a Homological Stability Theorem of the following form.

\begin{statement}\label{state:ReqHomStab}
There are constants $E_0$, $E_1$, and $N$, depending only on the abelian group $A$, such that
\begin{equation*}
H_d(\mathsf{Hn}_{G,n-N}^c(\bC)^{an};\bQ) \cong H_d(\mathsf{Hn}_{G,n}^c(\bC)^{an};\bQ)
\end{equation*}
for all $n \geq E_0 + E_1 \cdot d$.
\end{statement}

The space $\mathsf{Hn}_{G,n}^c(\bC)^{an}$ parametrises connected branched $G$-covers of $\mathbb{C}$ with $n$ branch points, the monodromy around each of which lies in the class $c \subset G$. Recording the set of $n$ distinct branch points gives a map
\begin{equation}\label{eq:HurwitzAsCover}
\pi: \mathsf{Hn}_{G,n}^c(\bC)^{an} \lra \mathsf{C}_n(\bC)^{an},
\end{equation}
and the fibre over a point $\{x_1, x_2, \ldots, x_n\} \in \mathsf{C}_n(\bC)^{an}$ is given by the set of isomorphism classes of connected principal $G$-bundles over $\bC \setminus \{x_1, x_2, \ldots, x_n\}$ whose monodromy around each $x_i$ lies in $c$; in particular this map is a covering space, and one may try to identify it. I mentioned above that the fundamental group of $\mathsf{C}_n(\bC)^{an}$ is Artin's braid group on $n$ strands, $\beta_n$. This braid group acts on $c^n = c \times c \times \cdots \times c$ by
\begin{equation}\label{eq:BraidAct}
\sigma_i * (g_1, g_2, \ldots, g_n) = (g_1, \ldots, g_{i-1}, g_ig_{i+1}g_i^{-1}, g_i,  g_{i+2}, \ldots, g_n),
\end{equation}
where $\sigma_i$ is the elementary braid that carries the $i$-th strand over the $(i+1)$-st. If  $c' \subset c^n$ denotes the subset of tuples which collectively generate $G$, and $G\backslash c'$ denotes the quotient of this set by the action of $G$ by conjugation of all elements, then the covering space \eqref{eq:HurwitzAsCover} may be identified as that given by the $\beta_n$-set $G\backslash c'$. From Theorem~\ref{thm:BraidComplexity} it follows that $\mathsf{Hn}_{G,n}^c(\bC)^{an}$ admits a cell structure with $2^n \cdot |G \backslash c'| \leq (2|c|)^n$ cells, providing the required Complexity Theorem in this case.

In the rest of this exposition I will explain the proof of Statement~\ref{state:ReqHomStab}.

\section{Spaces of marked branched covers}

Ellenberg, Venkatesh, and Westerland introduce the following topological model for Hurwitz spaces.

\begin{defi}
For $n>0$, $G$ a group, and $c \subset G$ a conjugation-invariant subset, let $\Hur_{G,n}^c$ denote the set of tuples $(t, \xi, f)$ where 
\begin{enumerate}
\item $t \in (0,\infty)$, 
\item $\xi$ is a configuration of $n$ distinct unordered points in $(0,t) \times (0,1)$, and 
\item $f : [0,t] \times [0,1] \setminus \xi \to BG$ is a continuous map sending $\sqcup := (\{0,t\} \times [0,1]) \cup ([0,t] \times \{0\})$ to the basepoint, and sending a small loop around each point of $\xi$ to a loop in $BG$ representing an element of $c \subset G$.
\end{enumerate}

Let $\CHur_{G,n}^c \subset \Hur_{G,n}^c$ denote the subset of those $(t,\xi,f)$ such that $f$ is $\pi_1$-surjective.
\end{defi}

One thinks of the data $(\xi, f)$ as describing a branched cover of $[0,t] \times [0,1]$, with branch points $\xi$ and monodromy given by $f$. There is a forgetful map 
\begin{equation}\label{eq:TopHurwitzAsCover}
\pi : \Hur_{G,n}^c \lra \mathrm{Conf}_n
\end{equation}
to the space of pairs $(t,\xi)$ as above, and a homotopy equivalence $\mathrm{Conf}_n \simeq \mathsf{C}_n(\bC)^{an}$. 

\begin{minipage}[t]{0.35\textwidth}
\vspace{0ex}
  \centering
  \includegraphics[width=\textwidth]{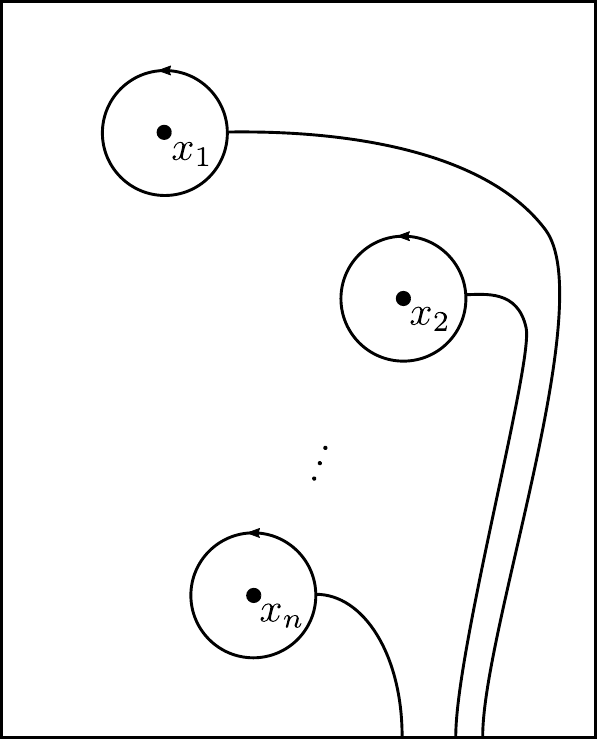}
\captionof{figure}{}
\label{fig:stdcurves} 
\end{minipage}\hfill
\begin{minipage}[t]{0.6\textwidth}
\hspace{2ex} The space $\Hur_{G,n}^c$ can be topologised so that the map $\pi$ is a fibration, whose fibre over $(t,\xi)$ is the space of maps
$$f : ([0,t] \times [0,1] \setminus \xi, \sqcup) \lra (BG, pt)$$
sending a small loop around each point of $\xi$ to a loop in the conjugation-invariant subset $c$. Recording the monodromy around $n$ loops based on $\sqcup$ as shown in Figure~\ref{fig:stdcurves} gives a homotopy equivalence between this mapping space and $c^n$. Thus, recalling that $\mathrm{Conf}_n \simeq \mathsf{C}_n(\bC)^{an}$ is an Eilenberg--MacLane space for the braid group $\beta_n$, the map \eqref{eq:TopHurwitzAsCover} may be identified up to homotopy with the Borel construction
\[c^n \hcoker \beta_n := c^n \times_{\beta_n} E\beta_n \lra B\beta_n,\]
with action as in \eqref{eq:BraidAct}.
\end{minipage}

\subsection{Relation to Hurwitz spaces}

If $c$ is a single conjugacy class and $c' \subset c^n$ denotes those tuples which generate $G$, then the above identification restricts to a homotopy equivalence between $\CHur_{G,n}^c$ and the homotopy orbit space $c' \hcoker \beta_n$. In the last section $\mathsf{Hn}_{G,n}^c(\bC)^{an}$ was identified with $(G\backslash c') \hcoker \beta_n$. To relate that with the spaces described here, note that the action of $G$ on itself by conjugation induces an action on $BG$, and hence an action on $\Hur_{G,n}^c$ by postcomposition of the maps $f$. Preferring to work with homotopy quotients, this gives a map
$$G \backslash\!\!\backslash \CHur_{G,n}^c \simeq G \backslash\!\!\backslash c' \hcoker \beta_n \lra G \backslash c' \hcoker \beta_n \simeq \mathsf{Hn}_{G,n}^c(\bC)^{an}$$
which induces an isomorphism on rational homology (as $G \backslash\!\!\backslash c'  \to G \backslash c' $ does, because the stabilisers of this action are finite groups).

Thus in order to prove Statement~\ref{state:ReqHomStab} it will be enough to prove the following, which is the second part of Theorem 6.1 of~\cite{EVW} (I will also prove the first part along the way). Recall that $A$ is a finite abelian $\ell$-group, $G = A \rtimes \bZ^\times$ where $\bZ^\times$ acts by inversion, and $c$ is the conjugacy class of all involutions in $G$.

\begin{theo}\label{thm:ReqHomStab2}
There are constants $E_0$, $E_1$, and $N$, depending only on $A$, such that there are $G$-equivariant isomorphisms
\begin{equation*}
H_d(\CHur_{G,n-N}^c;\bQ) \cong H_d(\CHur_{G,n}^c;\bQ)
\end{equation*}
for all $n \geq E_0 + E_1 \cdot d$.
\end{theo}

Statement~\ref{state:ReqHomStab} is immediately deduced from this by taking $G$-coinvariants.

\subsection{Multiplicative structure}\label{sec:MultStr}

For the discussion so far I could have worked with a simpler model of $\Hur_{G,n}^c$ given by setting $t=1$ everywhere. The advantage of allowing $t$ to vary is that it provides
$$\Hur_{G}^c := \coprod_{n \geq 0} \Hur_{G,n}^c$$
 with the structure of a topological monoid, where I declare $\Hur_{G,0}^c = \{1\}$ to be the unit and define the multiplication between terms with $n>0$ via the formula
$$(t, \xi, f) \cdot (t', \xi', f') = (t+t', \xi \amalg (\xi' + (t,0)), f'')$$
where
$$f''(x, y) = \begin{cases}
f(x, y) & \text{ if } 0 \leq x \leq t,\\
f'(x-t, y) & \text{ if } t \leq x \leq t+t'.
\end{cases}$$

\subsection{The ring of components}\label{sec:Components}

A consequence of this multiplicative structure is that for a field $k$ one can form a graded $k$-algebra $R = \bigoplus_{n \geq 0} R_n$ with $R_n := H_0(\Hur_{G,n}^c;k)$. By the homotopy equivalence $\Hur_{G,n}^c \simeq c^n \hcoker \beta_n$ which I have discussed, the vector space $R_n$ has a basis given by the quotient set $c^n/\beta_n$. I will write such equivalence classes as tuples $[g_1, g_2, \ldots, g_n]$, bearing in mind that for any $i$ the symbol $[g_1, \ldots, g_{i-1}, g_{i} g_{i+1} g_{i}^{-1}, g_i, g_{i+2}, \ldots g_n]$ represents the same element. Multiplication in the ring $R$ is given by concatenation of such tuples.

The \emph{global monodromy} of an element $[g_1, g_2, \ldots, g_n] \in
c^n/\beta_n$ is the product $g_1 g_2 \cdots g_n$ in~$G$, which is well-defined. If a tuple $[g_1, g_2, \ldots, g_n]$ has trivial global monodromy, then by applying elements of the braid group one sees that
$$[g_1, g_2, \ldots, g_n, h_1, \ldots, h_m] = [h_1, \ldots, h_m, g_1, g_2, \ldots, g_n]$$
for any $[h_1, \ldots, h_m] \in c^m/\beta_m$. In particular, a tuple $[g_1, g_2, \ldots, g_n]$ with trivial global monodromy lies in the centre of the graded ring $R$.

The goal of this section is to establish a strong structural result about the ring $R$, assuming the following property.

\begin{defi}\label{def:NonSplitting}
A pair $(G,c)$ of a group $G$ and a conjugacy class $c \subset G$ has the \emph{non-splitting property} if $c$ generates $G$ and if for each $H \leq G$ the set $c \cap H$ is either empty or is a conjugacy class of $H$.
\end{defi}

For $G = A \rtimes \bZ^\times$, with $A$ an abelian group of odd order and $\bZ^\times$ acting by inversion, and $c$~being the conjugacy class of all involutions in $G$, the pair $(G,c)$ has the non-splitting property.

\begin{prop}[Lemma 3.5 of~\cite{EVW}]\label{prop:DefU}
Suppose that $(G,c)$ has the non-splitting property and $|G|$ is a unit in the field $k$. Then there are natural numbers $N_0 \geq N$ and a $U \in R_N$ such that the maps
\begin{align*}
U \cdot - : R_{n-N} &\lra R_n\\
- \cdot U : R_{n-N} &\lra R_{n}
\end{align*}
are isomorphisms for all $n \geq N_0$.
\end{prop}

The main tool for proving this proposition is the following result of Conway and Parker, and Fried and V{\"o}lklein~\cite{FV}. A proof is included as Proposition 3.4 of~\cite{EVW}.

\begin{lemm}\label{lem:FV}
Let $G$ be a finite group, $c$ a conjugacy class of $G$, and $g \in c$. For sufficiently large $n$, every $n$-tuple $(g_1, g_2, \ldots, g_n) \in c^n$ of elements which generate $G$ is equivalent under the action of $\beta_n$ to a tuple $(g, g'_2, \ldots, g'_n)$ where $g'_2, \ldots, g'_n$ generate $G$.
\end{lemm}

\begin{proof}[Proof of Proposition~\ref{prop:DefU}]
Write $|g|$ for the order of $g \in G$. For a natural number $D$ one may form the element
$$U_D := \sum_{g \in c} [g]^{|g|D}$$ 
in the graded ring $R$, having grading $|g|D$. The element $[g]^{|g|}$ has trivial global monodromy so lies in the centre of $R$, and hence $U_D$ does too. The element $U$ will be $U_D$ for a large enough $D$, so in fact the left- and right-multiplication maps will be equal.

For a subgroup $H \leq G$, let $S_n(H) \subset c^n/\beta_n$ denote the subset of those $[g_1, \ldots, g_n]$ such that the group generated by $\{g_1, \ldots, g_n\}$ is $H$.

\vspace{1ex}

\noindent\textbf{Claim.} There is a $D \geq 0$ and a $n_0 \geq |g|D$ such that for every $n \geq n_0$, every $H \leq G$, and every $g \in c \cap H$, the function
$$[g]^{|g|D} \cdot - : S_{n-|g|D}(H) \lra S_n(H)$$
is a bijection, and this bijection is independent of $g$.

\begin{proof}[Proof of Claim]
First I will establish the claim for $D=1$ and without the ``independent of $g$" clause. If $c \cap H = \emptyset$ there is nothing to show; otherwise $c \cap H$ is a conjugacy class in $H$, so one can apply Lemma~\ref{lem:FV} to $(H, c \cap H)$, showing that $[g] \cdot - : S_{n-1}(H) \to S_n(H)$ is surjective for all large enough $n$. Iterating this, $[g]^{|g|} \cdot - : S_{n-|g|}(H) \to S_n(H)$ is surjective for all large enough $n$. As these are finite sets, it is in fact bijective for all large enough $n$.

I will now address the ``independent of $g$" clause. For $g_1, g_2 \in G$ and all large enough~$n$ the maps $[g_1]^{|g_1|} \cdot -$ and $[g_2]^{|g_2|} \cdot -$ are both bijections, so there are permutations
$$([g_1]^{|g_1|} \cdot -) \circ ([g_2]^{|g_2|} \cdot -)^{-1} : S_n(H) \lra S_n(H)$$
for all large enough $n$ (which are compatible with respect to $[g_1] \cdot - : S_n(H) \to S_{n+1}(H)$). If $D^H_{g_1, g_2}$ denotes the lowest common multiple of the orders of these permutations for all large enough $n$, which is finite by the compatibility property, then because $[g_i]^{|g_i|}$ commutes with any $[h_1, \ldots, h_m]$ it follows that 
$$[g_1]^{|g_1|D^H_{g_1,g_2}} \cdot - =[g_2]^{|g_2|D^H_{g_1,g_2}} \cdot - : S_{n-|g|D^H_{g_1, g_2}}(H) \lra S_n(H)$$
for all large enough $n$. Taking $D$ to be the product of $D_{g_1, g_2}^H$ over all $H \leq G$ and $g_1, g_2 \in c \cap H$ then has the required property.
\end{proof}

Choose $U=U_D$ with the $D$ provided by this claim. Filter the graded vector space $R$ by the subspaces $R^{\geq m}$ spanned by those tuples $[g_1, \ldots, g_n]$ which generate a subgroup of $G$ of order at least $m$. Multiplication by $U$ preserves this filtration, so induces a map after taking associated graded. One has
$$(R^{\geq m}/R^{\geq m+1})_{n-|g|D} = \bigoplus_{\substack{H \leq G \\ |H|=m}} k\{S_{n-|g|D}(H)\}.$$
Multiplication by $[g]^{|g|D}$ on the term $k\{S_{n-|g|D}(H)\}$ is 0 if $g \not\in H$, as then $g$ together with $H$ generate a group of order strictly larger than $m$. For all $g \in c \cap H$ and all large enough $n$ multiplication by $[g]^{|g|D}$ induces \emph{the same} isomorphism to $k\{S_{n}(H)\}$. Thus for all large enough $n$ the map
$$U \cdot - : (R^{\geq m}/R^{\geq m+1})_{n-|g|D} \lra (R^{\geq m}/R^{\geq m+1})_{n}$$
induces $|c \cap H|$ times an isomorphism, which is an isomorphism as $|c \cap H|$ divides $|G|$ so is a unit in $k$ by assumption. As the filtration of $R$ by $R^{\geq m}$'s is finite, the result follows.
\end{proof}

\section{Aside: Group-completion and delooping}

The material in this section does not appear in the work of Ellenberg, Venkatesh, and Westerland, though it is related to parts of their withdrawn preprint~\cite{EVW2} and Corollary~\ref{cor:StabHom} is proved there by different means. However, I find the point of view taken here clarifying, and it has informed the exposition I am giving of their results. I discovered the essential idea of Section~\ref{sec:delooping} in 2009, when Westerland sent me a draft of~\cite{EVW}. 

\subsection{Calculus of fractions and group-completion}

The multiplicative structure on $\Hur_G^c$ is certainly not commutative, even up to homotopy. One way to see this is to note that sending $(t,\xi,f)$ to the homotopy class of the restriction of $f$ to $[0,t] \times \{1\}$ defines a monoid homomorphism
$$\mu : \Hur_G^c \lra G,$$
with image the subgroup of $G$ generated by $c$, which need not be commutative. Now $\Hur_G^c$ is equipped with a left $G$-action, called $*$, induced by the conjugation action of $G$ on itself, and if the codomain of $\mu$ is considered as the left $G$-space $G^{ad}$ then the map $\mu$ is $G$-equivariant. This provides the data of a ``$G$-crossed space", cf.\ Remark~\ref{rem:GCrossedSpace} for this notion and the proof of the following lemma.

\begin{lemm}\label{lem:CalcFrac}
The maps $\Hur_G^c \times \Hur_G^c \to \Hur_G^c$ given by
\begin{align*}
(a,b) &\longmapsto a \cdot b\\
(a, b) &\longmapsto (\mu(a) *b) \cdot a
\end{align*}
are homotopic.
\end{lemm}

This lemma implies that the localisation $H_*(\Hur_{G}^c;k)[\pi_0(\Hur_{G}^c)^{-1}]$ may be constructed by right fractions. Thus the Group-Completion Theorem, in the form proved by Quillen~\cite[Theorem Q.4]{FM}, applies and yields a ring isomorphism
$$H_*(\Hur_{G}^c;k)[\pi_0 (\Hur_{G}^c)^{-1}] \cong H_*(\Omega B \Hur_{G}^c;k)$$
with the homology of the homotopical group-completion $\Omega B \Hur_{G}^c$.

\subsection{The delooping of $\Hur_{G}^c$ is the rack space of $c$}\label{sec:delooping}

There is a well-understood principle for producing models for deloopings of geometric monoids such as $\Hur_{G}^c$, originated by Segal~\cite{Segal}, which in this case identifies $B\Hur_{G}^c$ up to homotopy equivalence with the space $L$ of pairs $(\xi, f)$ consisting of 
\begin{enumerate}
\item a finite subset $\xi \subset \bR \times (0,1)$, and
\item a continuous map $f : (\bR \times [0,1] \setminus \xi, \bR \times \{0\}) \to (BG, *)$ sending a small loop around each point of $\xi$ to a loop in the conjugacy class $c$.
\end{enumerate}
However, the topology on this space is not what one first thinks! Rather, it is given a topology of convergence on compact subsets of $\bR \times [0,1]$, meaning that points of $\xi$ are allowed to move towards $\pm \infty$ in the $\bR$-direction and then vanish. In particular the cardinality of $\xi$ is not locally constant on this space.

Inside $L$ is the subspace $L^{red}$ consisting of those pairs $(\xi, f)$ for which the projection map $\xi \to [0,1]$ is injective, i.e.\ such that each line $\bR \times \{s\}$ contains at most one point of $\xi$.

\begin{lemm}
The inclusion $L^{red} \to L$ is a weak homotopy equivalence.
\end{lemm}
\begin{proof}
For each compact subset $K$ of $L$ there is an $\epsilon>0$ so that each interval $(-\epsilon,\epsilon) \times {s}$ contains at most one point of $\xi$. The subset $K$ then deforms into $L^{red}$ by stretching $(-\epsilon, \epsilon)$ to $\bR$, pushing off to $\pm \infty$ any points of $\xi$ outside $(-\epsilon,\epsilon) \times [0,1]$ and forgetting the map $f$ outside this region. This deformation preserves the subspace~$L^{red}$.
\end{proof}

(This argument might seem suspect at first; it is the unusual topology on $L$ which allows it. This kind of argument is often known as ``scanning".)

The space $L^{red}$ has a stratification by number of branch points. The $n$-th open stratum $S_n L^{red}$ has an evident map $S_n L^{red} \to \bR^n \times c^n$ given by recording the first coordinates of the $n$ branch points as well as their monodromies taken around the $n$ loops indicated in Figure~\ref{fig:stdcurves}. Taken together these yield a map $f: L^{red} \to |c^\bullet|$ to the geometric realisation of the (semi-)cubical space (also known as a $\square$-space) having $c^n$ as its set of $n$-cubes, and having face maps
\begin{equation}\label{eq:FaceMaps}
\begin{aligned}
d_i^0(g_1, g_2, \ldots, g_n) &= (g_1, g_2, \ldots, g_{i-1}, g_{i+1}, \ldots, g_n)\\
d_i^1(g_1, g_2, \ldots, g_n) &= (g_i^{-1} g_1 g_i, g_i^{-1} g_2 g_i, \ldots,  g_i^{-1} g_{i-1} g_i, g_{i+1}, \ldots, g_n).
\end{aligned}
\end{equation}
The map $f$ is easily seen to be a weak homotopy equivalence, by induction over strata. The geometric realisation $|c^\bullet|$ is precisely the \emph{rack space} $Bc$ of the rack\footnote{Recall that a \emph{rack} is a set $X$ with a binary operation $(a,b) \mapsto a^b$ for which each $(-)^b$ is an automorphism of sets-with-a-binary-operation. Ours is the prototypical example: a conjugacy class of some group, acting on itself by conjugation.} $c$, as defined by Fenn, Rourke, and Sanderson~\cite{FRSTrunk}.

\begin{prop}
There is a weak equivalence $B\Hur_{G}^c \simeq Bc$.
\end{prop}

In particular, the fundamental group of $B\Hur_{G}^c$ is the universal enveloping group of the conjugacy class $c$, and the homology of $B\Hur_{G}^c$ is the rack homology of $c$.

The following was proved by other means in~\cite{EVW2}.

\begin{coro}\label{cor:StabHom}
If $c$ is a conjugacy class which generates $G$, then each path component of $\Omega B\Hur_{G}^c$ has the rational homology of $S^1$.
\end{coro}
\begin{proof}
I first claim that the morphism of racks $\pi: c \to \{*\}$ induces an isomorphism on rational rack homology. Rack homology of $X$ over $k$ only makes use of the free $k$-module on $X$, $k\{X\}$, with its structure of a braided vector space. The map $1 \mapsto \frac{1}{|c|}\sum_{g \in c} g : \bQ \to \bQ\{c\}$ is a morphism of braided vector spaces which splits $\bQ\{\pi\} : \bQ\{c\} \to \bQ$, so $\pi$ is split surjective on rational rack homology.

On the other hand Etingof and Gra{\~n}a~\cite{EG} have shown that the dimension of the $d$-th rational rack homology is precisely $m^d$, where $m$ is the number of orbits of the rack acting on itself. If $c$ is a conjugacy class which generates $G$ then its action on itself has a single orbit, so its rational rack homology is 1-dimensional in each degree: by the above, $\pi$ must then induce an isomorphism on rational rack homology.

Writing $\mathrm{Conf} := \coprod_{n \geq 0} \mathrm{Conf}_n$ for the monoid made out of the configuration spaces as in \eqref{eq:TopHurwitzAsCover}, its delooping is given by the rack space of the trivial rack. Thus the map $B\Hur_G^c \to B\mathrm{Conf}$ is a rational homology equivalence, so $\Omega B\Hur_G^c \to \Omega B\mathrm{Conf} \simeq \Omega^2 S^2$ is a rational homology equivalence on each path-component. Finally, each path component of $\Omega^2 S^2$ has the rational homology of $S^1$.
\end{proof}

\section{$A$-homology and its Koszul complex}\label{sec:Koszul}

In this section and the rest of this exposition I will pass from working in topology to working in higher algebra, more specifically working with algebras and modules in the category of chain complexes. The topological monoid structure on $\Hur_G^c$ constructed in Section~\ref{sec:MultStr} endows its complex of singular $k$-chains
\begin{equation}\label{eq:DefA}
{A} := C_*(\Hur_{G}^c;k) = \bigoplus_{n \geq 0} C_*(\Hur_{G,n}^c;k)
\end{equation}
with the structure of a differential graded algebra (dga). I will consider it as an augmented dga with augmentation $\epsilon : A \to k$ given by projection to the summand $C_*(\Hur_{G,0}^c;k)$ followed by the canonical augmentation of this based chain complex. In this way $k$ has the structure of an $A$-bimodule.

Furthermore, the dga $A$ comes with an additional $\bN$-grading, given by the direct sum decomposition \eqref{eq:DefA}. From now on I will usually work in the category $\mathsf{Ch}_k^\bN$ of chain complexes of $k$-modules equipped with an additional $\bN$-grading, which formally means the category of functors $\mathsf{Fun}(\bN, \mathsf{Ch}_k)$. I will form tensor products in the graded sense. For an object $X \in \mathsf{Ch}_k^\bN$ I will write $X_n$ for the chain complex of grading $n$. For homology, I will write
$$H_{n,d}(X) := H_d(X_n).$$
Thus $H_{n,d}(A) = H_d(\Hur_{G,n}^c;k)$, and the $\bN$-graded dga structure on $A$ make these homology groups into a bigraded ring. In particular, the graded ring I have been denoting by $R$ is simply $H_{*,0}(A)$.

\begin{defi}
For a graded left $A$-module $M$ the \emph{$A$-module homology} $H^A_{n,d}(M)$ is the homology of the derived tensor product $k \otimes^\bL_A M$. More explicitly, it is the homology of the two-sided bar construction $B(k,A,M)$ formed in the category $\mathsf{Ch}_k^\bN$.
\end{defi}

The goal of this section is to develop a tool for the efficient calculation of $A$-module homology, by calculating the Koszul dual of $(A, \epsilon)$. By this I mean the dg coalgebra given by the two-sided bar construction $B(k,A,k)$ where $k$ is considered as a left or right $A$-module via $\epsilon$. The result is remarkably simple.

\begin{theo}\label{thm:KosDual}
One has
$$H_{n,d}(B(k,A,k)) = \begin{cases}
k\{c\}^{\otimes n} & \text{ if $n=d$}\\
0 & \text{ otherwise}.
\end{cases}$$
As a coalgebra this is the quantum shuffle coalgebra on the braided vector space $k\{c\}$.
\end{theo}
\begin{proof}
Write $B_\epsilon(A) = B(k,A,k)$ where $k$ is considered as an $A$-module via $\epsilon : A \to k$. Let $\epsilon_{geom} : A \to k$ be the alternative augmentation induced by the map of spaces $\Hur_{G}^c \to *$. If $A$ is filtered by its $\bN$-grading and $k$ is given the constant filtration then $\epsilon_{geom}$ is a filtered map, and its associated graded is $\epsilon : A \to k$. Thus there is a corresponding filtration of $B_{\epsilon_{geom}}(A)$ with associated graded $B_\epsilon(A)$.

Now $B_{\epsilon_{geom}} (A) \simeq C_*(B\Hur_{G}^c;k)$, and the filtration in question is induced by the filtration of the classifying space $B\Hur_{G}^c$ by number of branch points. The equivalences $B\Hur_{G}^c \simeq L^{red} \simeq Bc$ from Section~\ref{sec:delooping} are of filtered spaces when $Bc$ is given the filtration by cubical skeleta, so the associated graded of $B_{\epsilon_{geom}}(A)$ in grading $n$ is a pointed space homotopy equivalent to $S^n \wedge c^n_+$. The homology calculation follows.

The coproduct is given by the map on associated graded induced by the Serre diagonal of the (semi-)cubical set $Bc$, namely
$$\Delta(-) = \sum_{\substack{p+q=n \\ \sigma \in \mathrm{Sh}_{p,q}}} {\mathrm{sign}(\sigma)} \, d^0_{\sigma(1)} \cdots d^0_{\sigma(p)}(-) \otimes d^1_{\sigma(p+1)} \cdots d^1_{\sigma(p+q)}(-),$$
where the sum is taken over all decompositions $p+q = n$ and all $(p,q)$-shuffles. Using the formulas \eqref{eq:FaceMaps} for the face maps of this (semi-)cubical set, this gives the quantum shuffle coproduct for the braided vector space $k\{c\}$.
\end{proof}

A graded left $A$-module $M$ has a canonical filtration by $A$-submodules given by its grading, and the associated graded is isomorphic to $M$ as a graded chain complex, but the $A$-module structure is now that induced via the augmentation $\epsilon : A \to k$. Thus there is a corresponding filtration of $B(k,A,M)$ with associated graded $B(k,A,k) \otimes_k M$, and so a spectral sequence
\begin{equation}
E^1_{n,p,q}(M) = k\{c\}^{\otimes n+q} \otimes H_{-q,p-n}(M) \Longrightarrow H^A_{n,p+q}(M),
\end{equation}
with differentials $d^r : E^r_{n,p,q} \to E^r_{n,p+(r-1), q-r}$. As $M$ is an $A$-module each graded vector space $H_{*,d}(M)$ is a module over the graded ring $R=H_{*,0}(A)$ and as $H_{1,0}(A) = k\{c\}$ one in particular obtains action maps $k\{c\} \otimes H_{n-1,d}(M) \to H_{n,d}(M)$. By an analysis similar to the second part of Theorem~\ref{thm:KosDual} the $d^1$-differential of this spectral sequence can be identified in terms of this structure as
$$d^1(g_1 \otimes \cdots \otimes g_q \otimes [m]) = \sum_{i=1}^q (-1)^i g_1 \otimes \cdots \otimes g_{i-1} \otimes g_{i+1} \otimes \cdots \otimes g_q \otimes (g_i)^{g_{i+1} \cdots g_q} \cdot [m].$$
However, I shall not need to make use of this explicit formula.

If $M$ is \emph{discrete} (i.e.\ supported in homological degree 0) then an $A$-module structure on $M$ is the same as an $R$-module structure on $M$, induced along the quotient map $A \to H_{*,0}(A)=R$. In this case the above spectral sequence only has a $d^1$-differential and hence $H^A_{n, d}(M) \cong  E^2_{n,n,d-n}(M)$. Ellenberg, Venkatesh, and Westerland discover the chain complex $(E^1_{*,*,*}(M), d^1)$ differently to the way I have done so here, and they denote it $\mathcal{K}_{*,*}(M)$ and call it a ``Koszul-like complex"; from the point of view taken here it is precisely the Koszul complex for computing the $A$-module homology of a discrete $A$-module $M$.

\begin{rema}
The dual calculation to that of Theorem~\ref{thm:KosDual} has been made by Ellenberg, Tran, and Westerland~\cite{ETW}, who show that the $\mathrm{Ext}$-algebra of the quantum shuffle algebra on $k\{c\}$ agrees with $H_{*,*}(A)$.
\end{rema}

\begin{rema}\label{rem:GCrossedSpace}
There is another perspective on Theorem~\ref{thm:KosDual} which may be clarifying, involving the notion of ``$G$-crossed spaces", cf.~\cite[Section 4.2]{FreydYetter}. Consider the category $\mathsf{Top}^G/G^{ad}$ of $G$-spaces $X$ equipped with a $G$-equivariant map $\mu_X : X \to G^{ad}$. This category has a monoidal structure in which $(X, \mu_X) \otimes (Y, \mu_Y)$ is given by the $G$-space $X \times Y$ with the reference map
$$X \times Y \overset{\mu_X \times \mu_Y}\lra G^{ad} \times G^{ad} \overset{\cdot}\lra G^{ad}.$$
This monoidality admits a braiding via the formula
\begin{align*}
b_{X,Y} : X \times Y &\lra Y \times X\\
(x,y) &\longmapsto (\mu_X(x) * y , x).
\end{align*}

A conjugacy class $c \subset G$ defines an object $c$ of $\mathsf{Top}^G/G^{ad}$, with $G$-action given by conjugation and $\mu_c$ given by the inclusion $c \subset G^{ad}$. It is then almost tautological, given that $\Hur_{G,n}^c \simeq c^n \hcoker \beta_n$, that there is an equivalence of $E_1$-algebras
$$\Hur_{G}^c \simeq E_2(c),$$
where the free $E_2$-algebra on $c$ is formed in the braided monoidal category $\mathsf{Top}^G/G^{ad}$. This incidentally provides the proof of Lemma~\ref{lem:CalcFrac}.

From this point of view Theorem~\ref{thm:KosDual} simply records the well-known fact that the bar construction of an augmented free $E_k$-algebra on an object $X$ is the free $E_{k-1}$-algebra on the suspension of $X$ (taken in the category of pointed objects): in this case giving the free $E_1$-algebra on $S^1 \wedge c_+$.
\end{rema}

\section{The Regularity Theorem for $A$-homology}

In this section I will describe the most technical, and certainly most surprising, step in Ellenberg, Venkatesh, and Westerland's argument. From the point of view I am taking here, of modules over the dga $A$ and $A$-homology, the result can be interpreted as asserting that discrete $A$-modules which are finitely presented have finite Castelnuovo--Mumford regularity (in a slightly non-standard sense, in which a grading is scaled).

In this section, for a $\bN$-graded $k$-module $V$ I will write
$$\deg V := \inf\{k \in \bN \cup \{\infty\} \, | \, V_n=0 \text{ for all } n > k\}.$$
Recall that for a left $A$-module $M$ I defined $H_{n,d}^A(M) = H_{n,d}(k \otimes^\bL_A M)$, and write $h^A_d(M) := \deg H_{*,d}^A(M)$. The regularity theorem is then as follows, where as mentioned in the last section I consider $R$-modules as discrete $A$-modules via the quotient map $A \to H_{*,0}(A)=R$.

\begin{theo}[Theorem 4.2 of~\cite{EVW}]\label{thm:RegularityA}
Suppose that $(G,c)$ is non-splitting, as in Definition~\ref{def:NonSplitting}. Then there is a constant $B_0$ depending on $(G,c)$ such that for any left $R$-module $M$ one has
$$h^A_d(M) \leq \max(h^A_0(M), h^A_1(M)) + B_0 \cdot (d-1)$$
for all $d \geq 1$.
\end{theo}

To prove Theorem~\ref{thm:RegularityA}, Ellenberg, Venkatesh, and Westerland first consider the analogous regularity question for the graded ring $R$ instead of the graded dga $A$. To state this, for a left $R$-module $M$ let me write $H^R_{n,d}(M) := H_{n,d}(k \otimes^\bL_R M)$, and $h^R_d(M) := \deg H_{*,d}^R(M)$.

\begin{prop}[Proposition 4.10 of~\cite{EVW}]\label{prop:RegularityR}
Suppose that $(G,c)$ is non-splitting. Then there is a constant $B_1$ depending on $(G,c)$ such that for any left $R$-module $M$ one has
$$h^R_d(M) \leq   \max(h^R_0(M), h^R_1(M)) + B_1 \cdot (d-1).$$
for all $d \geq 1$.
\end{prop}

The following argument departs from that given by Ellenberg, Venkatesh, and Westerland, but to me seems considerably simpler.

\begin{proof}
Let $N_0 \geq N$ and $U \in R(N)$ be as in Proposition~\ref{prop:DefU}. Consider the homotopy cofibre (or mapping cone) $R \hcoker U$ of the map\
$$U \cdot - : k[N,0] \otimes R \lra R,$$
where $k[N,0]$ denotes the $\bN$-graded chain complex consisting of $k$ in grading $N$ and homological degree $0$. By Proposition~\ref{prop:DefU} this has $H_{n,d}(R \hcoker U)=0$ for all $n \geq N_0$ and all~$d$. One can consider $R \hcoker U$ as the derived tensor product $k \otimes ^\bL_{k[U]} R$, and can repeat this discussion with the module $M$ instead of $R$ to obtain $M \hcoker U \simeq k \otimes ^\bL_{k[U]} M$. Consider then
$$M \hcoker U \simeq k \otimes ^\bL_{k[U]} M \simeq (k \otimes ^\bL_{k[U]} R) \otimes_R^\bL M$$
and filter $k \otimes ^\bL_{k[U]} R \simeq R\hcoker U$ by its grading, to obtain a spectral sequence
\begin{equation}\label{eq:SSeqMmodU}
E^1_{n,p,q} = \bigoplus_{a+b=p+q} H_{-q, a}(R \hcoker U) \otimes H_{n+q, b}^R(M) \Longrightarrow H_{n,p+q}(M \hcoker U),
\end{equation}
which is depicted  in Figure~\ref{fig:SS} and has differentials $d^r : E^r_{n,p,q} \to E^r_{n, p+r-1, q-r}$. Both $R \hcoker U$ and $M \hcoker U$ have homology concentrated in homological degrees 0 and 1, because they are homotopy cofibres of maps between discrete objects.

\begin{minipage}[t]{0.4\textwidth}
\vspace{0ex}
\centering
  \includegraphics[width=\textwidth]{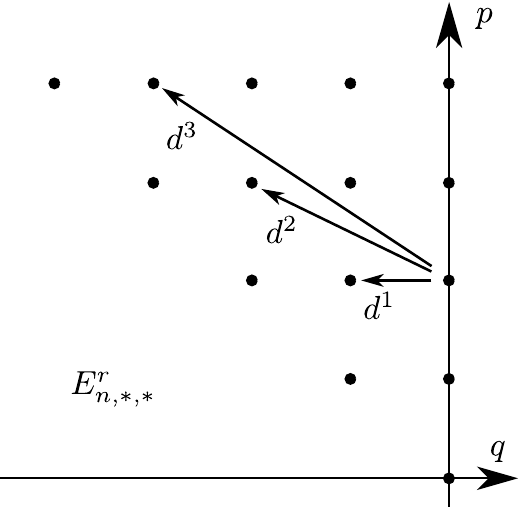}
\captionof{figure}{}
\label{fig:SS} 
\end{minipage}\hfill
\begin{minipage}[t]{0.55\textwidth}
\hspace{2ex} The claim clearly holds for $d=1$. Consider the group $E^1_{n,d,0}$ for $d \geq 2$. If elements of this group survived the spectral sequence then they would contribute to $H_{n,d}(M \hcoker U)$, but this vanishes so for $d \geq 2$ the group $E^1_{n,d,0}$ must die in the spectral sequence. The possible targets of differentials out of this group are $E^r_{n,d+(r-1), -r}$ with $r \geq 1$, as shown in Figure~\ref{fig:SS}. These are subquotients of
\begin{equation*}
\begin{aligned}
&H_{r, 0}(R \hcoker U) \otimes H_{n-r, d-1}^R(M)\\
&\quad\quad\quad \oplus H_{r, 1}(R \hcoker U) \otimes H_{n-r, d-2}^R(M),
\end{aligned}
\end{equation*}
which are zero if $r \geq N_0$ or if $n-r > h_{d-1}^R(M)$ and $n-r > h_{d-2}^R(M)$.  Thus, by induction on $d$,
\end{minipage}
\noindent as long as 
$\displaystyle n > r + \max(h^R_0(M), h^R_1(M)) +  B_1 \cdot (d-2) \text{ for all } r < N_0,$
or in other words as long as
$$n > N_0-1 + \max(h^R_0(M), h^R_1(M)) +  B_1 \cdot (d-2),$$
then there are no possible targets for differentials starting at $E^1_{n,d,0}$, so this group must vanish. But $H_{0,0}(R\hcoker U) \neq 0$, so if $E^1_{n,d,0}$ vanishes then so does $H_{n,d}^R(M)$. Choosing $B_1 := N_0-1$, it follows that $\deg H_{*,d}^R(M) \leq  \max(h^R_0(M), h^R_1(M)) + B_1 \cdot (d-1)$ as required.
\end{proof}

In order to deduce Theorem~\ref{thm:RegularityA} from Proposition~\ref{prop:RegularityR} one requires control of the $A$-homology of $R$, which is provided by the following lemma.

\begin{lemm}[Proposition 4.12 of~\cite{EVW}]\label{lem:Reg}
Suppose that $(G,c)$ is non-splitting. There is a constant $B_2$ depending on $(G,c)$ such that $h_d^A(R) \leq B_2 + d$. Furthermore~$h_0^A(R)=0$.
\end{lemm}
\begin{proof}
Let $N_0 \geq N$ and $U \in R(N)$ be as in Proposition~\ref{prop:DefU}. Consider the induced map
$$B(k,A, -\cdot U) : k[N,0] \otimes B(k,A,R) \lra B(k,A,R).$$

On the one hand this map is nullhomotopic, as each map $B(k,A,- \cdot [x])$ with $x \in \Hur_{G,N}^c$ having $N>0$ is. This is because of the functorial calculus of fractions explained in Lemma~\ref{lem:CalcFrac}, which shows that $- \cdot [x] :  R \to R$ is equal to an automorphism of $R$ followed by $[x] \cdot -$. But left multiplication by $[x]$ is simplicially nullhomotopic on this bar construction, as it is induced by the left $A$-module structure on $R$.

On the other hand, let us calculate the induced map $- \cdot U : H_{n-N,d}^A(R) \to H_{n,d}^A(R)$ using the Koszul complex for $A$-homology described in Section~\ref{sec:Koszul}, the relevant portion of which is
\begin{equation*}
\begin{tikzcd} 
k\{c\}^{\otimes d+1} \otimes R_{n-N-d-1} \rar \dar{-\cdot U}& k\{c\}^{\otimes d} \otimes R_{n-N-d}\rar \dar{-\cdot U}& k\{c\}^{\otimes d-1} \otimes R_{n-N-d+1} \dar{-\cdot U}\\
k\{c\}^{\otimes d+1} \otimes R_{n-d-1} \rar& k\{c\}^{\otimes d} \otimes R_{n-d}\rar& k\{c\}^{\otimes d-1} \otimes R_{n-d+1}.
\end{tikzcd}
\end{equation*}
The vertical maps are all isomorphisms as long as $n-d-1 \geq N_0$, by Proposition~\ref{prop:DefU}, so the induced map on homology in the middle position is an isomorphism in this range of degrees too.

Putting the results of the last two paragraphs together shows that the degree of $H_{*,d}^A(R)$ is at most $N_0+1 + d$, so the claimed result holds with $B_2 := N_0+1$.

For the addendum, note that $H_{*,0}^A(-)$ preserves epimorphisms, there is an epimorphism $A \to R$, and $H_{*,0}^A(A)$ is 1-dimensional supported in grading 0.
\end{proof}

\begin{proof}[Proof of Theorem~\ref{thm:RegularityA}]
Write 
$$B(k,A,M) = B(k,A,R) \otimes_R M \simeq B(k,A,R) \otimes^\bL_R M$$
and filter $B(k,A,R)$ by its grading to obtain a spectral sequence
$$E^1_{n,p,q} = \bigoplus_{a+b=p+q} H_{-q, a}^A(R ) \otimes H_{n+q, b}^R(M) \Longrightarrow H_{n,p+q}^A(M).$$
This again looks like Figure~\ref{fig:SS}.

Let me first show how to estimate $h_i^R(M)$ in terms of $h_i^A(M)$ for $i \leq 1$. In the case $p+q=0$ there can be no differentials leaving $E^r_{n,0,0}$, so it must vanish if $n > h_0^A(M)$. But $H_{0, 0}^A(R) \neq 0$ so $E^1_{n,0,0}$ contains $H_{n, 0}^R(M)$ as a summand, which must then vanish too, hence $h_0^R(M) \leq h_0^A(M)$. In the case $p+q=1$ differentials leaving $E^r_{n,1,0}$ go to $E^r_{n,r,-r} = H_{r, 0}^A(R ) \otimes H_{n-r, 0}^R(M)$, which vanishes if either $r >0$ (by the addendum in Lemma~\ref{lem:Reg}) or if $n-r > h_0^R(M)$ (so certainly if $n-r > h_0^A(M)$). Thus $E^1_{n,1,0}$ vanishes if both $n > h_0^A(M)$ and $n > h_1^A(M)$, but again $E^1_{n,1,0}$ contains $H^R_{n,1}(M)$ as a summand, giving $h_1^R(M) \leq \max(h_0^A(M), h_1^A(M))$.

The spectral sequence gives
$$h^A_d(M) \leq \max_{a+b=d}\{ h^A_a(R) + h^R_b(M)\}.$$
If $a=0$ then the relevant term is
\begin{align*}
h^A_0(R)+ h_d^R(M) &\leq 0+\max\{h_0^R(M),h_1^R(M)\} + B_1 \cdot (d-1) \\
&\leq \max\{h_0^A(M),h_1^A(M)\} + B_1 \cdot (d-1)
\end{align*}
using the addendum in Lemma~\ref{lem:Reg}, Proposition~\ref{prop:RegularityR}, and the above. If $a>0$ then the relevant term is
\begin{align*}
h^A_a(R) + h^R_b(M) &\leq B_2 + a + \max\{h_0^A(M),h_1^A(M)\} + B_1 \cdot (b-1) \\
&= B_2 + (1-B_1)a + \max\{h_0^A(M),h_1^A(M)\} + B_1 \cdot (d-1)
\end{align*}
using Lemma~\ref{lem:Reg} and the above. In either case, as long as $B_0 \geq \max\{B_1, B_2+1\}$ this is $\leq \max\{h_0^A(M),h_1^A(M)\} + B_0 \cdot (d-1)$. In particular, by comparison with the proofs of Proposition~\ref{prop:RegularityR} and Lemma~\ref{lem:Reg} it follows that $B_0 = N_0+2$ suffices.
\end{proof}

For homological stability one is interested in the range of degrees in which multiplication by $U$ induces an isomorphism on an $R$-module $M$, in other words the degree of $H_{*,d}(M \hcoker U)$ for $d=0$ and $d=1$. This can be estimated in terms of the $A$-homology of~$M$ as follows.

\begin{coro}\label{cor:AHomImpliesFP}
For a left $R$-module $M$ one has $\deg(H_{*,0}(M\hcoker U)) \leq h_0^A(M)+N_0$ and $\deg(H_{*,1}(M\hcoker U)) \leq \max(h_1^A(M), h_0^A(M))+N_0$.
\end{coro}
\begin{proof}
Use the spectral sequence \eqref{eq:SSeqMmodU} from the proof of Proposition~\ref{prop:RegularityR} again. 

The terms which could contribute to $H_{n,0}(M\hcoker U)$ are $H_{-q,0}(R \hcoker U) \otimes H_{n+q,0}^R(M)$ with~$q \leq 0$. Now $H_{-q,0}(R \hcoker U)=0$ for $-q \geq N_0$, and by the proof of Theorem~\ref{thm:RegularityA} $H_{n+q,0}^R(M)=0$ for $n+q > h_0^A(M)$, since $h_0^A(M) \geq h_0^R(M)$. Thus these terms all vanish for $n > h_0^A(M)+N_0$.

The terms which could contribute to $H_{n,1}(M\hcoker U)$ are
$$H_{-q,1}(R \hcoker U) \otimes H_{n+q,0}^R(M) \oplus H_{-q,0}(R \hcoker U) \otimes H_{n+q,1}^R(M) \text{ for } q \leq 0.$$
Now $H_{-q,i}(R \hcoker U)=0$ for $-q \geq N_0$, and by the proof of Theorem~\ref{thm:RegularityA} $H_{n+q,1}^R(M)=0$ for $n+q > \max(h_1^A(M), h_0^A(M))$, since $\max(h_1^A(M), h_0^A(M)) \geq h_1^R(M)$. Thus these terms all vanish for $n > \max(h_1^A(M), h_0^A(M))+N_0$.
\end{proof}

\section{Proof of homological stability}

Given the preparations now made, the proof of homological stability for the spaces of marked branched covers $\Hur_{G,n}^c$ is very formal. In this context homological stability in degree $d$ means showing that the maps
$$U \cdot - : H_{n-N,d}(A) \lra H_{n,d}(A)$$
are isomorphisms for all large enough $n$, and by Corollary~\ref{cor:AHomImpliesFP} to show this it is enough to bound the degrees of $H^A_{*,0}(H_{*,d}(A))$ and  $H^A_{*,1}(H_{*,d}(A))$ appropriately. By the regularity theorem it is equivalent to bound the degrees of all $H^A_{*,p}(H_{*,d}(A))$ appropriately, which is better suited to an inductive proof.

\begin{prop}
$\deg H^A_{*,p}(H_{*,d}(A)) \leq  B_0(2d+p)$.
\end{prop}
\begin{proof}
Filtering the $A$-module $A$ by its grading induces a filtration of $B(k,A,A) \simeq k$ with associated graded $B(k,A,k) \otimes A$, and so an associated spectral sequence
$$E^1_{n,p,q} = k\{c\}^{\otimes n+q}\otimes H_{-q,p-n}(A) \Longrightarrow H_{n,p+q}(B(k,A,A))$$
which converges to 0 for $(n,p+q) \neq (0,0)$. As discussed in Section~\ref{sec:Koszul}, taking homology with respect to the $d^1$-differential computes the $A$-module homology of the graded $R$-modules $H_{*,p-n}(A)$, considered as discrete modules, so the spectral sequence takes the form
$$E^2_{n,p,q} = H^A_{n, n+q}(H_{*, p-n}(A)) \Longrightarrow H_{n,p+q}(B(k,A,A)),$$
with differentials $d^r : E^r_{n,p,q} \to E^r_{n, p+r-1, q-r}$.

\begin{minipage}[t]{0.47\textwidth}
\vspace{0ex}
\centering
  \includegraphics[width=\textwidth]{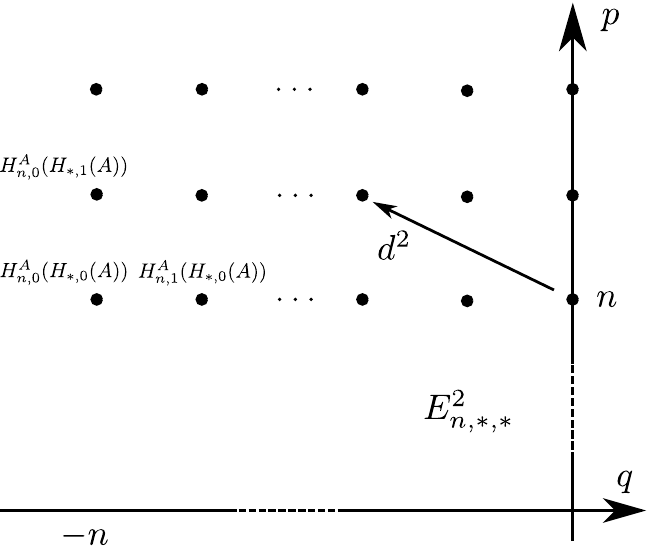}
\end{minipage}\hfill
\begin{minipage}[t]{0.48\textwidth}
\hspace{2ex} Let us prove the claim in the proposition by induction on $d$: it holds for $d<0$, so suppose it holds for all $d'<d$. 

\hspace{2ex} Firstly, $E^2_{n, d+n, -n} = H^A_{n,0}(H_{*, d}(A))$, but $E^\infty_{n, d+n, -n}=0$ for $n>0$ so $E^2_{n, d+n, -n}$ must die during the spectral sequence. There are no differentials leaving this group, and the possible differentials entering this group come from $E^r_{n, d+n-(r-1), -n+r}$ with $r \geq 2$, which is a subquotient of $H^A_{n, r}(H_{*, d-(r-1)}(A))$. 

\hspace{2ex} This gives the estimate
\end{minipage}

\begin{align*}
h_0^A(H_{*,d}(A)) &\leq \max_{r \geq 2}\{h_r^A(H_{*,d-(r-1)}(A))\}\\
& \leq \max_{r \geq 2}\{B_0(2(d-(r-1))+r)\}\\
& \leq B_0(2d).
\end{align*}

Secondly, $E^2_{n, d+n, -n+1} = H^A_{n,1}(H_{*, d}(A))$, which for $n>0$ must also die during the spectral sequence. There are again no differentials leaving this group, and the possible differentials entering this group come from $E^r_{n, d+n-(r-1), -n+r+1}$ with $r \geq 2$, which is a subquotient of $H^A_{n, r+1}(H_{*, d-(r-1)}(A))$. Thus
\begin{align*}
h_1^A(H_{*,d}(A)) &\leq \max_{r \geq 2}\{h_{r+1}^A(H_{*,d-(r-1)}(A))\}\\
& \leq \max_{r \geq 2}\{B_0(2(d-(r-1))+r+1)\}\\
& \leq B_0(2d+1).
\end{align*}

Finally, by the regularity theorem we have $h_p^A(H_{*, d}(A)) \leq  B_0(2d+1) + B_0 (p-1) = B_0(2d+p)$ as required.
\end{proof}

As described above, and recalling that $H_{n,d}(A) = H_d(\Hur_{G,n}^c;k)$, by Corollary~\ref{cor:AHomImpliesFP} the cokernel of
\begin{equation}\label{eq:NonConnStab}
U \cdot - : H_{d}(\Hur_{G,n-N}^c;k) \lra H_{d}(\Hur_{G,n}^c;k)
\end{equation}
vanishes for $n> B_0(2d) +N_0$, and its kernel vanishes for $n> B_0(2d+1)+N_0$, which proves homological stability for the spaces $\Hur_{G,n}^c$. Theorem~\ref{thm:ReqHomStab2} claims an analogous homological stability theorem for the spaces $\CHur_{G,n}^c$ of connected branched covers, and furthermore claims that the stabilisation map is $G$-equivariant. By its definition the element $U$ is $G$-invariant and so the map $U \cdot -$ is $G$-equivariant, and hence Theorem~\ref{thm:ReqHomStab2} is a consequence of the following. 

\begin{coro}
Let $(G,c)$ be non-splitting. There are constants $E_0$ and $E_1$, depending only on $(G,c)$, such that 
\begin{equation*}
U \cdot - : H_d(\CHur_{G,n-N}^c;\bQ) \lra H_d(\CHur_{G,n}^c;\bQ)
\end{equation*}
is an isomorphism for all $n \geq E_0 + E_1 \cdot d$.
\end{coro}
\begin{proof}
For each subgroup $H \leq G$ write $A^{(H)} \leq A = C_*(\Hur_G^c;\bQ)$ for the sub-($\bN$-graded chain complex) given by the chains on those path-components of $\Hur_G^c$ represented by $[g_1, g_2, \ldots, g_n]$ where the subgroup generated by $\{g_i\}$ is $H$. Let $A^{\geq m}$ denote the dg ideal of $A$ given by $\bigoplus_{|H| \geq m} A^{(H)}$. Then $A=A^{\geq 0}$, and $A^{\geq |G|} = A^{(G)} = C_*(\CHur_G^c;\bQ)$, so that it is required to show that $U \cdot - : H_{n-N,d}(A^{\geq |G|}) \to H_{n,d}(A^{\geq |G|})$ induces an isomorphism in a linear range of bidegrees. 

The map $U \cdot -$ preserves the filtration $A^{\geq \bullet}$, and I will show that it induces an isomorphism on $H_{*,*}(A^{\geq m})$ in a linear range of degrees by induction on $m$, the base case $m=0$ being given by \eqref{eq:NonConnStab}. On the associated graded
$$H_{*,*}(A^{\geq m}/A^{\geq m+1}) = \bigoplus_{|H|=m} H_{*,*}(A^{(H)})$$
the map $U \cdot -$ decomposes as $\bigoplus_{|H|=m} U^{(H)} \cdot -$, where $U^{(H)} = \sum_{g \in c \cap H} [g]^{|g|D}$. By induction on $|G|$ one may suppose that for each proper subgroup $H < G$ the map $U^{(H)} \cdot - : H_{n-N,d}(A^{(H)}) \to H_{n,d}(A^{(H)})$ is an isomorphism for all $n \geq E_0^H + E_1^H \cdot d$. The result then follows from an iterated application of the 5-lemma.
\end{proof}

\section{Outlook}

In later work Ellenberg, Venkatesh, and Westerland~\cite{EVW2} explained what the required topological input would be to strengthen Theorem~\ref{thm:Main} from a statement about the limit as $q \to \infty$ of the upper and lower densities $\delta^\pm(q)$ to a statement about the actual values of the upper and lower densities for all large enough  $q$. In fact the preprint~\cite{EVW2} attempted to prove this strengthening, but came across the following difficulty.

\vspace{2ex}

The discussion in Section 5 shows that for $V := \prod_{g \in c}[g]^{|g|} \in R_{|c| |g|}$, on one hand
$$H_{*}(\CHur_{G}^c;k)[V^{-1}] \cong H_*(\Omega B\Hur_G^c;k)$$
and on the other hand each path-component of $\Omega B\Hur_G^c$ has the same rational homology as $S^1$. So \emph{if one knew} the map
\begin{equation*}
V \cdot - : H_d(\CHur_{G,n}^c;\bQ) \lra H_{d}(\CHur_{G,n+|c||g|}^c;\bQ)
\end{equation*}
was an isomorphism as long as $n \geq E_0 + E_1 \cdot d$ for some constants $E_0$ and $E_1$ depending only on the abelian $\ell$-group $A$, then it would follow that each path-component of $\CHur_{G,n}^c$ had the same rational homology as $S^1$ in degrees $d \leq \frac{n-E_0}{E_1}$. 

One could then employ the argument described in Section~\ref{sec:StabHomArg} to deduce that
$$\lim_{n \to \infty} \frac{\# \mathsf{Hn}_{G,n}^c(\bF_q)}{q^n} = 1-q^{-1}$$
for all large enough $q$ which are good for $\ell$. The discussion in Section 2 shows that
$$\frac{\sum_{L \in \mathfrak{S}_n} m_A(L)}{|\mathfrak{S}_n|} = \frac{\# \mathsf{Hn}_{G,n}^c(\bF_q)}{q^n-q^{n-1}},$$
giving that
$$\lim_{n \to \infty} \frac{\sum_{L \in \mathfrak{S}_n} m_A(L)}{|\mathfrak{S}_n|} = 1$$
for all large enough $q$ which are good for $\ell$. In the notation of Section~\ref{sec:CLHeur}, it would then follow that
$$\delta^+(q) = \delta^-(q) = \mu(A)$$
for all large enough $q$ which are good for $\ell$, with no need to take the limit as $q \to \infty$.

\vspace{2ex}

However, the above argument cannot yet be made because it is not yet known whether multiplication by $V$ is an isomorphism in any range of homological degrees, only that multiplication by $U$ is. These are quite different maps, and it seems very difficult to approach stability for $V$ by a variant of the method explained here. Ellenberg, Venkatesh, and Westerland's argument for $U$ proceeds by showing that it induces isomorphisms on homology for the larger spaces $\Hur_{G,n}^c$ of possibly disconnected branched covers, but these definitely do not enjoy homological stability with respect to $V$ (even their $0$-th homology fails to stabilise). In a different direction, if one replaces the monoid $\Hur_{G}^c$ by its submonoid $\CHur_{G}^c$ of connected branched covers then a statement as simple as Theorem~\ref{thm:KosDual} cannot possibly hold. However, in contrast to these difficulties, in every other situation in which I have seen homological stability it is maps akin to $V$ which induce isomorphisms. Homological stability of the spaces $\CHur_{G,n}^c$ with respect to the maps $V$ should serve as a guiding problem for mathematicians working in this subject.

\vspace{2ex}

\noindent\textbf{Acknowledgements.} I would like to thank Andrea Bianchi, Jordan Ellenberg, Manuel Kannich, Jack Thorne, and Nathalie Wahl for useful comments. The author was partially supported by the ERC under the European Union's Horizon 2020 research and innovation programme (grant agreement No.\ 756444), and by a Philip Leverhulme Prize from the Leverhulme Trust.

\bibliographystyle{smfalpha}
\bibliography{Hurwitz}  

\end{document}